\documentclass[reqno]{amsart}
\usepackage[utf8]{inputenc}
\usepackage[T1]{fontenc}
\usepackage{lmodern}
\usepackage[top=3cm, bottom=3cm, left=3cm, right=3cm]{geometry}
\usepackage{amsmath, amssymb, amsthm}
\usepackage{hyperref}
\usepackage{xcolor}
\usepackage{enumitem}
\usepackage{fancyhdr}
\usepackage{booktabs}
\usepackage[normalem]{ulem}

\newtheorem{theorem}{Theorem}[section]

\newtheorem{corollary}[theorem]{Corollary}
\newtheorem{proposition}[theorem]{Proposition}
\newtheorem{lemma}[theorem]{Lemma}

\theoremstyle{remark}

\usepackage{enumitem}
\usepackage{hyperref,amssymb}
\begin{document}

	\title[$\mathrm{CCS}$ groups]{Finite groups in which every proper characteristic subgroup is cyclic}
	\author[M. Damele]{Marco Damele}
	\address{Marco Damele, Facoltà di Scienze Matematiche, Fisiche e Naturali, University of Cagliari, Via Ospedale, 72, 09124 Cagliari, Italy} 
	\email{m.damele4@studenti.unica.it}
	\author[F. Mastrogiacomo]{Fabio Mastrogiacomo}
	\address{Fabio Mastrogiacomo, Dipartimento di Matematica ``Felice Casorati", University of Pavia, Via Ferrata 5, 27100 Pavia, Italy} 
	\email{fabio.mastrogiacomo01@universitadipavia.it}
	\subjclass[2020]{20E34}
	\keywords{$\mathrm{CCS}$ groups, characteristic subgroups, cyclic group}       
	\thanks{The second author is member of the GNSAGA INdAM research group and kindly acknowledges their support.}
	\maketitle	
	\begin{abstract}
Let $G$ be a finite, non-cyclic, non-characteristically simple group such that all its proper characteristic subgroups are cyclic. We call such a group a $\mathrm{CCS}$ group, short for \emph{Characteristic Cyclic Subgroups}. In this paper, we provide a complete classification of these groups.  

As a consequence, we obtain an alternative proof that any skew brace whose multiplicative group is cyclic of $p$-power order, with $p$ an odd prime, necessarily has a cyclic additive group. Moreover, we describe the multiplicative group of skew braces whose additive group is a solvable, non-nilpotent $\mathrm{CCS}$ group.

	\end{abstract}
	
	\section{Introduction}
	Let $\mathfrak{C}$ be a class of groups. A finite group $G$ is said to be $\mathfrak{C}$-critical, or minimal non-$\mathfrak{C}$, if $G \notin \mathfrak{C}$ while every proper subgroup of $G$ belongs to $\mathfrak{C}$. For instance, a minimal non-cyclic group is a finite non-cyclic group in which all proper subgroups are cyclic. Minimal non-$\mathfrak{C}$ groups have been extensively studied for different choices of $\mathfrak{C}$, with the aim of achieving a complete classification. Understanding the structure of such groups often provides valuable insight into the properties that characterize membership in $\mathfrak{C}$. The first substantial results in this direction were obtained in \cite{MM}, where the authors classified both minimal non-abelian and minimal non-cyclic groups. Subsequently, further classifications have been carried out for minimal non-nilpotent groups (the so-called "Schmidt groups") and for minimal non-supersolvable groups (see \cite{BER}, \cite{GL}, \cite{BE}, respectively). Ito (see \cite{Ito1951}) considered the minimal non-p-nilpotent groups for $p$ a prime, which turn out to be just the Schmidt groups. Robinson characterised in \cite{Robinson1969} the minimal non-$\mathrm{PST}$-groups, where a $\mathrm{PST}$-group is a group in which Sylow permutability is a transitive relation. In a related line of research, several authors have investigated groups not belonging to $\mathfrak{C}$ in which certain natural families of subgroups lie in $\mathfrak{C}$. A classical example is given by the $\mathrm{Z}$-groups, introduced by Suzuki in~\cite{S}, namely groups whose Sylow subgroups are all cyclic. More recently, in~\cite{DM} the authors classified the non-cyclic groups in which all normal subgroups are cyclic. In our terminology, a non-cyclic, non-simple group in which every proper normal subgroup is cyclic will be called an $\mathrm{NCS}$ group (Normal Cyclic Subgroups). In this paper we extend this perspective by introducing and studying a broader class of groups. Let $G$ be a finite group, and let $\mathcal{C}(G)$ denote the set of all proper characteristic subgroups of $G$. Suppose that $G$ is non-cyclic, that $\mathcal{C}(G)\neq \varnothing$ (i.e. $G$ is not characteristically simple), and that every subgroup in $\mathcal{C}(G)$ is cyclic. In this case, we say that $G$ is a $\mathrm{CCS}$ group (Characteristic Cyclic Subgroups). For example $D_{8}$, the dihedral group of order $8$, is a $\mathrm{CCS}$ group since its non-trivial characteristic subgroups are $C_4$ and $C_2$. 
   
	
	
	Our main result provides a complete classification of $\mathrm{CCS}$ groups.
	\begin{theorem}\label{thm:main}
		Let $G$ be a $\mathrm{CCS}$ group.
		Then, one of the following holds.
		\begin{itemize}
			\item[(i)] $G \cong p^{1+2n}_+$, where $p$ is an odd prime and $n$ is a positive integer.
			\item [(ii)]$G \cong 2^{1+2n}_+$ or $G\cong 2^{1+2n}_-$ for some positive integer $n$.
			\item [(iii) ]$G \cong 2^{1+2n}_+ \circ C_4$, where $\circ$ denotes the central product.
			\item[(iv)] $G$ is a dihedral group.
			\item[(v)] $G$ is a dicyclic group.
			\item[(vi)] $G$ has presentation
			\[
			\langle x,y \, | \, x^{mp} = y^p = 1, yxy^{-1} = x^k \rangle,
			\]
			where $m,k$ are positive integers, $p$ the smallest prime dividing the order of $G$, $(m,p)=1$, $k^p \equiv 1 \mod mp$, $k \not \equiv 1 \mod mp$, and $(k-1,m)=1$.
			\item[(vii)] $G$ has presentation 
			\[  \langle x,y \ | \ x^m=y^{p^{\alpha}}=1,yxy^{-1}=x^{k} \rangle,
			\]
			where $m,k$ are positive integers, $p$ the smallest prime dividing the order of $G$, $(m, k-1)=1$ and $k^{p}\equiv1 \mod m$.
			\item[(viii)]  $G$ is a perfect group, $Z(G)$ is cyclic and it is the unique maximal characteristic subgroup of $G$. 
			Moreover, if this happens, $G$ fits into a short exact sequence of the form $$1 \rightarrow C \rightarrow G \rightarrow S \times \cdots \times S$$ where $S$ is a non-abelian simple group and $C$ is a cyclic group isomorphic to a quotient of $H^{2}(S,\mathbb{Z}) \times \cdots \times H^{2}(S,\mathbb{Z})$.
		\end{itemize}
		Conversely, each of these groups is a $\mathrm{CCS}$ group.
	\end{theorem}
	 Our  theorem  naturally extends the $\mathrm{NCS}$ classification of \cite{DM}, where one requires all \emph{proper normal subgroups} to be cyclic. Here the requirement is stronger, all \emph{proper characteristic subgroups} must be cyclic, and this has two broad consequences. On the one hand, the entire \lq\lq metacyclicpart/ Z-group'' already present in \cite{DM} survives essentially unchanged: the same presentations (semidirect products with explicit arithmetic constraints) reappear as a natural subcase of the  $\mathrm{CCS}$ framework, and the perfect case remains tied to central extensions with cyclic kernel. On the other hand, the greater characteristic rigidity forces genuinely new families that do not show up in the $\mathrm{NCS}$ taxonomy: in particular, extraspecial $p$-groups (and certain $2$-extraspecial central products) and uniform dihedral/dicyclic families, all stable under the full automorphism group. \\
	 
	   The study of $\mathrm{CCS}$ groups, as a natural extension of $\mathrm{NCS}$ groups, is not only interesting in its own right but, as we will see, provides a useful framework for investigating skew left braces, an algebraic structure introduced to study set-theoretic solutions of the Yang-Baxter equation. Formally a \emph{skew left brace} is defined as a triple $(B,+,\cdot)$ such that $(B,+)$ and $(B,\cdot)$ are groups, and the following compatibility condition holds for all $a,b,c \in B$:
	  \begin{equation} \label{compatibility}
	  	a \cdot (b+c) = a \cdot b - a + a \cdot c,
	  \end{equation}
	  where $-a$ denotes the inverse of $a$ in the group $(B,+)$. The group $(B,+)$ is referred to as the \emph{additive group} of the skew brace, while $(B,\cdot)$ is called the \emph{multiplicative group}. Skew left braces are a generalization of left braces introduced by Rump in \cite{Rump2007bis}, namely skew braces with abelian additive group. The study of this object can be traced back to Drinfeld (\cite{Drinfeld1992}), who suggested investigating set-theoretic solutions of the Yang--Baxter equation $(\mathrm{YBE})$, namely pairs $(X,r)$ where $X$ is a set and $r \colon X \times X \to X \times X$ is a bijective map satisfying
	  \[
	  (r \times \mathrm{Id}_X) \circ (\mathrm{Id}_X \times r) \circ (r \times \mathrm{Id}_X)
	  =
	  (\mathrm{Id}_X \times r) \circ (r \times \mathrm{Id}_X) \circ (\mathrm{Id}_X \times r),
	  \]
	  where $\circ$ denotes composition of maps. A solution $(X,r)$ is called \emph{involutive} if
	  \[
	  r \circ r = \mathrm{Id}_{X \times X}.
	  \]
	First, Rump  (\cite{Rump2007bis}) introduced the notion of a left brace to study involutive non-degenerate solutions, proving that every left brace gives rise to a solution of the Yang--Baxter equation. Subsequently, it was shown in \cite{Bachiller2016} that every involutive non-degenerate solution arises from a left brace. To handle non-involutive solutions, Guarnieri and Vendramin \cite{GV17} introduced skew left braces, and it was later shown \cite{Bachiller2018} that all non-degenerate solutions can be obtained from skew left braces. 
One line of research in the study of skew left braces is the investigation of the relationship between the additive and the multiplicative groups of a skew left brace. In this direction, one of the main conjectures—originally posed by Vendramin in \cite{Vendramin2019}—states that a finite skew left brace whose additive group is solvable must have a solvable multiplicative group. 

This conjecture is known to hold when the additive group is nilpotent (see \cite{Smoktunowicz2018}) and for several other classes of skew braces (see, for example, \cite{Gorshkov2021}, \cite{Nasybullov2019}, \cite{Byott2024}, \cite{Tsang2019} and in a recent preprint on Lie skew braces \cite{DL}. More generally, it is of significant interest to understand the structure of the multiplicative group of a skew brace in terms of group-theoretical properties of its additive group. Conversely, determining structural properties of the additive group $(B,+)$ of a skew brace $B$, given group-theoretic properties of $(B,\cdot)$, seems in general much more difficult. For example, it is known that if the multiplicative group of a skew brace $B$ is nilpotent, then the additive group $(B,+)$ is solvable. Moreover, if the multiplicative group $(B,\cdot)$ is cyclic, then the additive group $(B,+)$ is supersolvable, and if $(B,\cdot)$ is abelian, then $(B,+)$ is metabelian. (see, for instance, \cite[Theorem 1.3]{Tsang2019}. 
Another classic result is the following:

\begin{theorem} \label{Th2}
	Let $(B,+,\cdot)$ be a skew brace. If $(B,\cdot)$ is cyclic of order $p^n$ for some odd prime $p$, then $(B,+) \simeq (B,\cdot)$.
\end{theorem}

A proof of Theorem~\ref{Th2} is given in \cite{Kohl1998}, where the result is obtained via detailed computations. In this paper we present a different proof based on Theorem 
\ref{thm:main}.The case $p=2$ is settled in \cite{Byott2007}.

To conclude we give a general description of skew braces with non-nilpotent non-perfect $\mathrm{CCS}$ additive group:

\begin{theorem}\label{thm:skbr}
	Let $(B,+,\cdot)$ be skew brace such that $(B,+)$ is a $\mathrm{CCS}$ group, with $|B| = p^{\alpha}p_2^{\alpha_2}\cdots p_t^{\alpha_t} = p^\alpha m$, where $2<p<p_2<\cdots<p_t$ are primes, and $\alpha,\alpha_2,\dots,\alpha_t$ are non-negative integers. Suppose that $(B,+)$ is non-nilpotent and non-perfect. Then, one of the following holds.
	\begin{itemize}
		\item The group $(B,\cdot)$ is nilpotent, and $(B,\cdot) \cong C_m \times Q$, where $Q$ is a $p$-group with a maximal cyclic subgroup.
		
		\item The group $(B,\cdot)$ is non-nilpotent, and $(B,\cdot) \cong K \rtimes Q$, where $K$ is a $\mathrm{Z}$-group of order $m$, and $Q$ is a $p$-group with a maximal cyclic subgroup.
	\end{itemize} 
\end{theorem}

As a consequence, Vendramin's conjecture holds when the additive group is a solvable $\mathrm{CCS}$ group.
\begin{corollary}\label{cor:vendraminCCS}
Let $(B,+,\cdot)$ be a finite skew brace such that $(B,+)$ is a solvable $\mathrm{CCS}$ group. Then the multiplicative group $(B,\cdot)$ is solvable.
\end{corollary}

\medskip
The paper proceeds as follows. Sections~\ref{sec:1}-\ref{sec:3} establish the classification of $\mathrm{CCS}$ groups (nilpotent, solvable non-nilpotent, and perfect cases), while ~\ref{sec:skew} uses this classification to study skew braces and prove Theorems~\ref{Th2} and~\ref{thm:skbr}.

More precisely, the proof of Theorem~\ref{thm:main} is structured according to the standard trichotomy (nilpotent / solvable non-nilpotent / perfect). We begin with the nilpotent case: in Section~\ref{sec:1} we prove that a nilpotent $\mathrm{CCS}$ group must be a non-abelian $p$-group (Lemma~\ref{lemma:pgrp}), and we then classify $\mathrm{CCS}$ $p$-groups. This classification produces the families appearing in Theorem~\ref{thm:main}(i)--(iii), and it also accounts for the $p$-group instances of parts (iv)--(v), namely the dihedral and quaternion cases.
We then move to the solvable but non-nilpotent situation. In Section~\ref{sec:2} we deal with the case in which $G'<G$. The crucial structural ingredient here is Theorem~\ref{sec2:thm:1}; once this is established, we derive the metacyclic presentations listed in Theorem~\ref{thm:main}(vi)--(vii), and we also obtain the dihedral and dicyclic families appearing in (iv)--(v).
Finally, we treat the perfect case. In Section~\ref{sec:3} we consider groups with $G=G'$ and we prove Theorem~\ref{thm:main}(viii). To conclude, in Section~\ref{sec:skew} we apply Theorem~\ref{thm:main} to skew braces, obtaining in particular the proofs of Theorems~\ref{Th2} and~\ref{thm:skbr}.

	\subsection*{Acknowledgements}
	The authors would like to thank Andrea Loi and Vicent Pérez-Calabuig for their helpful discussions and insightful comments, which greatly improved this work.
	
	\section{Notation and preliminaries}
	In this brief section, we introduce the notation and preliminaries that will be used throughout the paper.\\
	Let $p$ be a prime number and let $G$ be a $p$-group of order $p^{n}$ for some positive integer $n$. Let $i \in \{1, \ldots n\}$. We denote with $$\Omega_i(G) = \langle g\in G \, : \, g^{p^i}=1\rangle,$$ and with $$\mho_i(G) = \langle g^{p^i} \, : \, g \in G \rangle.$$ If $i=1$, we write $\Omega(G) = \Omega_1(G)$ and $G_p = \mho_1(G)$.  \\
	Recall that an extraspecial $p$-group is a $p$-group $G$ with $Z(G) = G' = \Phi(G) = C_p$, where $G' = [G,G]$ is the derived subgroup of $G$, and $\Phi(G)$ is the Frattini subgroup of $G$. For each prime $p$ and each positive integer $n$, there are two classes of extraspecial groups, denoted with $p^{1+2n}_+$ and $p^{1+2n}_-$. In the former case, the exponent of the group is $p$, in the latter is $p^2$.\\
	A $p$-group $G$ is called regular if for every $g,h \in G$, there exists $k \in \langle g,h\rangle' $ such that
	\[
	g^ph^p = (gh)^pk^p.
	\]
	We recall here two easy properties about regular $p$-groups. This is standard material in the theory of $p$-groups. We refer the reader, for example, to~\cite[Chapter~$1$]{LGM}.
	\begin{lemma}\label{lemma:regularpgrp}
		Let $G$ be a $p$-group. 
		\begin{itemize}
			\item If the nilpotency class of $G$ is strictly less than $p$, then $G$ is regular.
			\item If $G$ is regular, then
			\[
			[G:G_p] = |\Omega(G)|.
			\]
		\end{itemize}
	\end{lemma}
	Let now $G$ and $H$ be two groups and take $G_1 \leq Z(G)$ and $H_1 \leq Z(H)$. Let $\theta : G_1 \to H_1$ be an isomorphism. Then, we define the (external) central product of $G$ and $H$ as
	\[
	G \circ H = \frac{G\times H}{N},
	\]
	where $N = \{(g,h) \in G_1 \times H_1 \, | \, \theta(g) = h\}$. \\
	A group $G$ is said to be the internal central product of two subgroups $G_1$ and $G_2$ if $G = G_1G_2$, and $[G_1,G_2]=1$. 	It is easy to see that if $G$ is the external central product of $K$ and $H$, then $G$ is the internal central product of $G_1$ and $G_2$, where $G_1$ is the image of $K \times 1$ in the quotient group $(K \times H)/N$, where $N$ is defined as above, and $G_2$ is the image of $1 \times H$. Moreover $G_1 \cong K$ and $G_2 \cong H$. We refer the reader to~\cite{Go} for more details.\\

	We denote with $D_{2n}$ the dihedral group of order $2n$. Moreover, we denote with $Q_{2^n}$ the generalized quaternion group of order $2^n$, with presentation
	\[
	Q_{2^n} = \langle x , y \, | \, x^{2^{n-1}} = y^4 = 1, yxy^{-1}=x^{-1} \rangle.
	\]
	We then denote with $\mathrm{Dic}_{n}$ the dicyclic group of order $4n$. Recall that such a group has presentation
	\[
	\mathrm{Dic}_{n} = \langle a,b \, | \, a^{2n}=1, b^2=a^n, bab^{-1} = a^{-1} \rangle.
	\]
	Observe that if $n$ is a power of $2$, then $\mathrm{Dic}_n$ is a generalized quaternion group. Observe moreover that if $n=2^{k}m$ is even, the dicyclic group is the unique split extension of the cyclic group of order $m$ with the quaternion group of order $2^{k+2}$. \\
	Finally, we denote with $SD_{2^n}$ the semidihedral group, that is
	\[
	SD_{2^n} = \langle r, s \, | \, r^{2^{n-1}}=s^2=1, srs=r^{2^{n-2}-1}\rangle.
	\]
	\begin{lemma}\label{lemma:semidihedral}
		For any $n>1$, the group $SD_{2^n}$ is not a $\mathrm{CCS}$ group.
	\end{lemma}
	\begin{proof}
		Let $G$ be a finite group and $H \leq G$ be a subgroup of index $2$. Then, $G^2 \leq H$ and therefore $H/G^2$ is a subgroup of index at most $2$ in $G/G^2$. This means that the number of subgroups of index $2$ in $G$ is at most the number of subgroups of index $2$ in $G/G^2$.\\
		Now take $G = SD_{2^n} = \langle r,s \rangle$. Then, $G^2 \geq \langle r^2,s^2 \rangle = \langle r^2 \rangle$, and therefore $|G^2| \geq 2^{n-2}$, so that $|G/G^2| \leq 4$. In particular, $G/G^2$ (and therefore $G$) has at most $3$ subgroups of index $2$. Now it is easy to see that $H_1 = \langle r \rangle$, $H_2 = \langle r^2,s \rangle$ and $H_3 = \langle r^2, rs \rangle$ are three different subgroups of index $2$ in $G$, and therefore these are the unique subgroups of index $2$ in $G$. Moreover, we have that $H_1 \cong C_{2^{n-1}}$, $H_2 \cong D_{2^{n-1}}$ and that $H_3 \cong Q_{2^{n-1}}$. Therefore, these are characteristic subgroups of $G$, so that $G$ is not a $\mathrm{CCS}$ group.
	\end{proof}
	
	We define a Frobenius group as a group $G$ admitting a normal subgroup $N$ with the property that, for every non trivial element $n$ of $N$, $\mathbf{C}_G(n)\leq N$. This is not the standard definition typically found in the literature on permutation groups, but it is an equivalent one. (see, for example, \cite[Theorem~$6.4$]{Is}). \\
	If $G$ is a Frobenius group, the subgroup $N$ with the property that $\mathbf{C}_G(n)\leq N$ for every $n \in N$ is called \emph{Frobenius kernel} of $G$. A well known theorem states that if $G$ is a Frobenius group with Frobenius kernel $N$, then $N$ admits a complement $A$, so that $ G = N \rtimes A$. The subgroup $A$ is called \emph{Frobenius complement}.\\
	We will use the following property of Frobenius group.
	\begin{proposition}\cite[Corollary~$6.10$]{Is}\label{prop:frob}
		Let $G$ be a Frobenius group. Then, the Sylow $p$-subgroups of the Frobenius complement of $G$ are either cyclic or (if $p=2$) generalized quaternion.
	\end{proposition}
	
	We need two basic facts about supersolvable groups. 	Recall that a group $G$ is said to be supersolvable if it has a normal series with cyclic quotient, that is a series
	\[
	1 = G_0  \trianglelefteq G_1 \trianglelefteq G_2 \trianglelefteq \dots \trianglelefteq G_{n-1} \trianglelefteq G_n = G,
	\]
	where $G_i \trianglelefteq G$ and $G_i / G_{i-1}$ is cyclic, for all $i=1,\dots,n$.
	\begin{lemma}\cite[Lemma~$2.17$]{A}\label{lemmasupersolv1}
		Let $G$ be a finite group with $G'$ cyclic, where $G' = [G,G]$ is the commutator subgroup. Then $G$ is supersolvable.
	\end{lemma}
	\begin{lemma}\cite[Theorem~$4.24$]{C}\label{lemmasupersolv2}
		Let $G$ be a finite supersolvable group, and let $p$ be the smallest prime dividing the order of $G$. Then, the elements of order prime to $p$ form a normal $\pi$-Hall subgroup of $G$, where $\pi$ is the set of primes dividing the order of $G$ different from $p$.
	\end{lemma}
	
	Finally, we present an easy property of $\mathrm{CCS}$ groups. 	Observe that the property of $\mathrm{CCS}$ group usually does not pass to subgroup. The smallest example showing this is the extraspecial group $2^5_+$, which has a subgroup of shape $C_4 \times C_2$. However, the following holds.
	\begin{lemma}\label{lemma:charQuot}
		Let $G$ be a finite group and let $N$ be a characteristic subgroup of $G$. Let $H/N$ be a characteristic subgroup of $G/N$. Then, $H$ is a characteristic subgroup of $G$.
	\end{lemma}
	\begin{proof}
		Let $\varphi \in \mathrm{Aut}(G)$, and $\psi : G/N \to G/N$ defined by $\psi(xN) = \varphi(x)N$. This is an automorphism of $G/N$, since $N$ is characteristic in $G$, and thus  $\psi(H/N) = H/N$, implying that $H$ is characteristic in $G$.
	\end{proof}
	\begin{lemma}\label{lemma:quot}
		Let $G$ be a $\mathrm{CCS}$ group and let $N$ be a characteristic subgroup of $G$. Then, $G/N$ is either characteristically simple, or is $\mathrm{CCS}$ group.
	\end{lemma} 
	\begin{proof}
		It follows from Lemma~\ref{lemma:charQuot}.
	\end{proof}

	\section{Nilpotent $\mathrm{CCS}$ groups}\label{sec:1}
	In this section, we deal with nilpotent $\mathrm{CCS}$ groups. 
	\begin{lemma}\label{lemma:pgrp}
		Let $G$ be a nilpotent $\mathrm{CCS}$ group. Then, $G$ is a non-abelian $p$-group, for some prime $p$.
	\end{lemma}
	\begin{proof}
		Suppose that $G$ is not a $p$-group. Then, $G$ is the direct product of its Sylow subgroups, and each of these is characteristic in $G$. In particular, each Sylow subgroup is cyclic, so that $G$ itself is cyclic, a contradiction. Thus, $G$ is a non-cyclic $p$-group. Suppose that $G$ is abelian. Since $G$ is not cyclic, $\Omega(G)$ is a non-cyclic proper characteristic subgroup of $G$, implying that $\Omega(G) = G$, and hence $G$ is an elementary abelian $p$-group, again a contradiction. 
	\end{proof}
	The remaining of this section is devoted to classify non-abelian $\mathrm{CCS}$ $p$-groups. We will distinguish the case where $p$ is odd or $2$.\\
	We begin by assuming that $p$ is odd.
	\begin{proposition}
		Let $G$ be a non-abelian $p$-group, with $p>2$. Then, $G$ is a $\mathrm{CCS}$ group if, and only if, $G \cong p^{1+2n}_+$, for some $n \in \mathbb{N}$.
	\end{proposition}
	\begin{proof}
		Suppose that $G$ is a $\mathrm{CCS}$ group.	The Frattini subgroup $\Phi(G)$ is a characteristic subgroup, and then it is cyclic. From~\cite[Theorem~$2.4$]{FJO}, we see that $\Phi(G) \leq Z(G)$. Since $G/\Phi(G)$ is elementary abelian, we have that $G' \leq \Phi(G) \leq Z(G)$, and then,
		\[
		\frac{G}{Z(G)} \cong \frac{G/\Phi(G)}{\Phi(G)/Z(G)}.
		\]
		In particular, $G/Z(G)$ is elementary abelian, and thus it has exponent $p$. We now show that $G' \cong C_p$. Since $G' \leq Z(G)$, the nilpotency class of $G$ is $2<p$.  \cite[Theorem~$2.4$]{AKT} shows that 
		\[
		\exp(G') | \exp\left(\frac{G}{Z(G)}\right) = p.
		\]
		Therefore, $\exp(G')=p$ and $G' = C_p$. Observe now that $G/G'$ is an abelian $\mathrm{CCS}$ group, by Lemma~\ref{lemma:quot}. Therefore, $G/G'$ is elementary abelian, so that
		\[
		\exp(G) | \exp(G')\exp \left(\frac{G}{G'}\right) = p^2.
		\]
		Aiming for a contradiction, suppose that $\exp(G) = p^2$. Observe that $G$ is a regular $p$-group, since its nilpotency class is less then $p$, and thus
		\[
		[G:G_p] = |\Omega(G)|,		
		\]
		by Lemma~\ref{lemma:regularpgrp}. Take now $g \in G_p$. Then, $g = x_1^p \cdots x_k^p$, for some $x_1,\dots,x_k \in G$. However, $G_p$ is cyclic. In particular, $g^p = x_1^{p^2}\cdots x_k^{p^2} = 1$. This means that $G_p$ is a cyclic $p$-group of exponent $p$, and thus $G_p \cong C_p$. This means that $\Omega(G)$ is a maximal subgroup of $G$. Moreover, if $g \in G'$, then $g^p = 1$, leading to $G' \leq \Omega(G)$. If $G' = \Omega(G)$, then $|G| = p^2$, since $|G'|=p$ and $\Omega(G)$ is maximal. But this is impossible, since $G$ is not abelian. Then, $G' < \Omega(G)$, and $\Omega(G)/G'$ is a cyclic subgroup of $G/G'$, which is an elementary abelian group. Thus,
		\[
		\frac{\Omega(G)}{G'}\cong C_p.
		\]
		In particular, $|\Omega(G)| = p^2$ and $|G| = p^3$. In particular, $G$ is the extraspecial group $p^3_-$, and it is easy to see that this has a characteristic subgroup of shape $C_p \times C_p$, a contradiction. Therefore, $\exp(G) = p$. In particular, since $\exp(Z(G)) | \exp(G)$, we have that $Z(G) = C_p$, and therefore $G$ is an extraspecial group of exponent $p$, as claimed. Assume now that $G = p^{1+2n}_+$. It is a known fact that $G$ admits a unique proper characteristic subgroup, which is its center, and this is cyclic by definition.
	\end{proof}
	Assume now that $p=2$. Here, we do not still have that $\Phi(G) \leq Z(G)$. In particular, we need to consider the two cases separately.
	\begin{proposition}
		Let $G$ be a non-abelian $2$-group with $\Phi(G) \leq Z(G)$. Then, $G$ is a $\mathrm{CCS}$ group if, and only if, one of the following occurs.
		\begin{itemize}
			\item $G \cong 2^{1+2n}_\pm$.
			\item $G \cong 2^{1+2n}_+ \circ C_4$.
		\end{itemize}
	\end{proposition}
	\begin{proof}
		Suppose that $G$ is an extraspecial $2$-group of order $2^{1+2n}$. If $n=1$, then $G$ is either the dihedral group or the quaternion group, and these are both $\mathrm{CCS}$ groups. If $n>1$, it is known that $G$ has a unique proper characteristic subgroup, which is the center, and this is cyclic by definition. 
		Assume then that $G$ is a $\mathrm{CCS}$ group, and that $G$ is not extraspecial. Aiming for a contradiction, suppose that $\Omega(G) < G$. Thus, $\Omega(G)$ is a cyclic subgroup formed by all the elements of order $2$. Since $\Omega(G)$ is cyclic, we have that $\Omega(G) = C_2$, and so $G$ contains a unique element of order $2$. It is known that such a group must be either cyclic or the generalized quaternion group. Then, $G$ is the generalized quaternion group. However, this is a contradiction, since the Frattini subgroup of this group is not central. Thus, $\Omega(G) = G$, and by \cite[Proposition~$3.3$]{Bor} we obtain that $G = 2^{1+2n}_+ \circ C_4$.
		It remains to show that $G = 2^{1+2n}_+ \circ C_4$ is a  $\mathrm{CCS}$ group. We have $$G=(2^{2n+1}_{+} \times C_{4}) / N,$$ where $C_2 \cong N  \leq Z(2^{2n+1}_+)\times Z(C_4) = C_2 \times C_4$. In particular, we have $Z(G) = Z(2^{2n+1}_+)\times Z(C_4) /N \cong C_4$. We claim that every characteristic subgroup of $G$ lies in its center. By definition, we have that $G = G_1 G_2$, with $G_1 \cong 2^{1+2n}_+$ and $G_2 \cong C_4$, and $[G_1,G_2]=1$. Let $K$ be a characteristic subgroup of $G$, and write $$K = (K\cap G_1)(K\cap G_2).$$ 
		Of course $K\cap G_2 \leq G_2 = Z(G_2)$. We now show that $K\cap G_1 \leq Z(G_1)$. Let $\varphi \in \mathrm{Aut}(G_1)$. Then, the map $\psi : G \to G$ defined by 
		$$\psi(g_1g_2) = \varphi(g_1) g_2,$$
		for $g_1 \in G_1$ and $g_2 \in G_2$ is an automorphism of $G$ (this follows from the fact that $G_1$ and $G_2$ commute). In particular, $\psi(K) = K$. Take now $k_1 \in K \cap G_1$. Then, $\psi(k_1) \in K$. But $\psi(k_1)=\varphi(k_1) \in G_1$. Thus, $\varphi(k_1) \in K \cap G_1$, implying that $K \cap G_1$ is a characteristic subgroup of $G_1$. But $G_1 \cong 2^{1+2n}_+$, and its unique characteristic subgroup is its center. Thus, $K \cap G_1 \leq Z(G_1)$. This shows that $K \leq Z(G)$, and thus $K$ is cyclic.
	\end{proof}
	In the following proposition, we take into account the case of $2$-groups with Frattini subgroup not central, completing the classification of $\mathrm{CCS}$ $2$-groups.
	\begin{proposition}
		Let $G$ be a non-abelian $\mathrm{CCS}$ $2$-group of order $|G|=2^n$, with $\Phi(G) \not \leq Z(G)$. Then, either $G$ is a dihedral group, or $G$ is a generalized quaternion group.
	\end{proposition}
	\begin{proof}
		Suppose that $\Omega(G) < G$. By~\cite[Proposition~$3.3$]{Bor}, either $G$ is isomorphic to $Q_{2^n}$ or $G$ is isomorphic to $SD_{2^n}$. However, by Lemma~\ref{lemma:semidihedral}, $SD_{2^n}$ is not a $\mathrm{CCS}$ group, and so $G = Q_{2^n}$ in this case.\\
		Suppose then that $\Omega(G) = G$, so that $G$ is generated by elements of order $2$. Take $C = \mathbf{C}_G(\Phi(G))$. Since $\Phi(G)$ is not central, we have that $C < G$. Since $\Phi(G)$ characteristic, so is its centralizer. In particular, $C$ is a cyclic group. The proof of $(5)$ of~\cite{BERGER} shows that $\Phi(G) = G^2$, $[G:C]=2$, and that $G=\langle h,C\rangle$, where $h$ is any element outside $C$. Thus, since $C$ is cyclic, $G$ has a minimal generating set of cardinality $2$. But every minimal generating set of a $p$-group has the same cardinality. Therefore, we can extract a generating set of two elements from the set of elements of order $2$. In particular, $G$ is a dihedral group.
	\end{proof}
	Summing up, we have the following classification of $\mathrm{CCS}$ $p$-groups.
	\begin{proposition}
		Let $G$ be a $p$-group, for some prime $p$. Then, $G$ is a $\mathrm{CCS}$ group if, and only if, one of the following holds.
		\begin{itemize}
			\item $p=2$, and $G$ is either a  dihedral group, or a generalized quaternion group, or $G \cong 2^{1+2n}_\pm$ or $G = 2^{1+2n}_+ \circ C_4$, for some $n>0$.
			\item $p$ is odd, and $G \cong p^{1+2n}_+$, for some $n>0$.
		\end{itemize}
	\end{proposition}
	\section{Non-nilpotent solvable $\mathrm{CCS}$ groups}\label{sec:2}
	Suppose now that $G$ is a non-nilpotent $\mathrm{CCS}$ group. In particular, $G' \leq G$ is a non trivial subgroup. For the remaining of this section, we suppose that $G' < G$, so that $G'$ is cyclic. Hence, by Lemma~\ref{lemmasupersolv1}, $G$ is supersolvable. The first result of this section gives a characterization of these groups.
	\begin{theorem}\label{sec2:thm:1}
		Let $G$ be a non-nilpotent $\mathrm{CCS}$ group with $G' < G$, and let $p$ be the smallest prime dividing the order of $G$. Then, 
		\[
		[G:\mathrm{F}(G)] = p,
		\]
		where $\mathrm{F}(G)$ is the Fitting subgroup of $G$. In particular, $\mathrm{F}(G)$ is a maximal subgroup of $G$.
	\end{theorem}
	\begin{proof}
		We have that $\mathbf{C}_G(\mathrm{F}(G)) = Z(\mathrm{F}(G)) = \mathrm{F}(G)$. In particular, $G/\mathrm{F}(G) \leq \mathrm{Aut}(\mathrm{F}(G))$. But the Fitting subgroup is a cyclic subgroup, and hence its automorphism group is abelian, and therefore $G/\mathrm{F}(G)$ is abelian. Let now $q$ be a prime number dividing $|G/\mathrm{F}(G)|$, and consider the subgroup $\Omega_q(G/\mathrm{F}(G))$ of the elements of order $q$. This is a characteristic subgroup of $G/\mathrm{F}(G)$. Suppose that $ \Omega_q(G/\mathrm{F}(G)) = H/\mathrm{F}(G)$, for some $H \leq G$. Suppose that $H \neq G$. In particular, by Lemma~\ref{lemma:charQuot}, $H$ is a characteristic subgroup of $G$, and so it is cyclic. This implies that $H \leq \mathrm{F}(G)$, and so $H = \mathrm{F}(G)$ and $\Omega_q(G/\mathrm{F}(G)) = 1$, a contradiction. In particular, $\Omega_q(G/\mathrm{F}(G)) = G/\mathrm{F}(G)$, implying that $G/\mathrm{F}(G)$ is an elementary abelian group of order $q^s$, for some $s$. Write now $|G| = p_1^{\alpha_1} p_2^{\alpha_2} \cdots p_t^{\alpha_t}$, where $p_1 < p_2 < \dots < p_t$. Since $G$ is supersolvable, by Lemma~\ref{lemmasupersolv2} the subset formed by all the elements of order coprime to $p_1$ is a characteristic subgroup of $G$, and so it is contained in $\mathrm{F}(G)$. In particular, $q=p_1$. Moreover, if $P_i$ is the Sylow $p_i$-subgroup of $G$, and $i\geq 2$, then $P_i \leq \mathrm{F}(G)$. In particular, $|\mathrm{F}(G)| = n = p_1^b p_2^{\alpha_2}\cdots p_t^{\alpha_t}$, for some $b < \alpha_1$, and $G/\mathrm{F}(G) = C_{p_1}^k$, where $k = a-b$. We claim that $k=1$. Aiming for a contradiction, suppose that there exists a non-nilpotent $\mathrm{CCS}$ group $G$ with $G'<G$ such that $G/\mathrm{F(G)} = C_{p_1}^k$, with $k>1$, where the orders of $G$ and $\mathrm{F}(G)$ are as before. Suppose moreover that $G$ is of minimum order with these properties. Observe that $\Phi(G) < \mathrm{F}(G)$. Indeed, if $\Phi(G) =\mathrm{F}(G)$, then $G' \leq \mathrm{F}(G) = \Phi(G)$, and thus $G$ is nilpotent, a contradiction. 
		We claim that $Z(G) = 1$. Firstly, take $A = G/\Phi(G)$. This is a solvable group, and thus $A' < A$. Moreover, $A$ is not abelian, because otherwise $G' \leq \Phi(G)$, which is impossible. Since $\Phi(G) < \mathrm{F}(G)$, we have
		\[
		\mathrm{F}(A) = \frac{\mathrm{F}(G)}{\Phi(G)}\neq 1.
		\]
		Moreover, $\mathrm{F}(A) < A$, because otherwise $\mathrm{F}(G) = G$. In conclusion, $A$ is not-nilpotent and $A$ is not characteristically simple. By Lemma~\ref{lemma:quot}, $A$ is a $\mathrm{CCS}$ group. Moreover,
		\[
		\frac{A}{\mathrm{F}(A)} \cong \frac{G}{\mathrm{F}(G)} = C_{p_1}^k.
		\] 
		By minimality of $G$, we have that $A=G$, so that $\Phi(G) =1$.\\
		Take now $B = G/Z(G)$. Since $G' \cap Z(G) \leq \Phi(G) = 1$, we have
		\[
		B' = \frac{G'Z(G)}{Z(G)} \cong \frac{G'}{G' \cap Z(G)} = G'.
		\]
		In particular, $B' \neq 1$ and $B' < B$, since $B$ is solvable. Moreover, $\mathrm{F}(B) = \mathrm{F}(G)/Z(G) \neq B$, because otherwise $\mathrm{F}(G) = G$. Moreover, if $\mathrm{F}(B) = 1$, then $\mathrm{F}(G) = Z(G)$ and thus $\mathrm{F}(G) \cap G' = 1$, implying that $G' = 1$, a contradiction. This shows that $B$ is not characteristically simple, and thus by Lemma~\ref{lemma:quot} $B$ is a $\mathrm{CCS}$ group. Moreover,
		\[
		\frac{B}{\mathrm{F}(B)} \cong \frac{G}{\mathrm{F}(G)} = C_{p_1}^k.
		\]
		By minimality of $G$, we have that $B=G$ and thus $Z(G) = 1$.\\ 
		Let $x$ be a non trivial element of $\mathrm{F}(G)$, and consider $\langle x \rangle \leq \mathrm{F}(G)$. Since $\mathrm{F}(G)$ is a characteristic subgroup of $G$, $\langle x \rangle$ is also a characteristic subgroup of $G$. Therefore, even its centralizer $\mathbf{C}_G(x)$ is characteristic in $G$. Since $G$ is a $\mathrm{CCS}$ group, either $\mathbf{C}_G(x) = G$ or $\mathbf{C}_G(x) \leq \mathrm{F}(G)$. However, the former case implies that $\langle x \rangle \leq Z(G)=1$, which is impossible. Thus, $\mathbf{C}_G(x)\leq \mathrm{F}(G)$. This shows that $G$ is a Frobenius group, with Frobenius kernel $\mathrm{F}(G)$ and complement 
		\[
		\frac{G}{\mathrm{F}(G)} = C_{p_1}\times \cdots \times C_{p_1}.
		\]
		By Proposition~\ref{prop:frob}, the Sylow $p_1$-subgroups of $C_{p_1}\times \cdots \times C_{p_1}$ are either cyclic or generalized quaternion, and this is a contradiction if $k>1$. Thus, $k=1$, and the proof is completed.

	\end{proof}
	Let now $G$ be a non-nilpotent $\mathrm{CCS}$ group with $G'\neq G$, of order $|G| = p_1^{\alpha_1}\cdots p_t^{\alpha_t}$, with $p_1 < p_2 < \dots p_t$ primes and $\alpha_i$ non-negative integers. Since, as we already said, $G$ is supersolvable, it admits a normal subgroup $N$ of order $m = p_2^{\alpha_1}\cdots p_t^{\alpha_t}$. This subgroup $N$ is also characteristic, and thus $N = C_m$. If $P$ is a Sylow $p_1$-subgroup of $G$, we have that $G = C_m \rtimes P$. Observe also that since $G/C_m = P$, $P$ is either an elementary abelian $p_1$-group, or a $\mathrm{CCS}$ $p_1$-group, by Lemma~\ref{lemma:quot}. Moreover, since $G/\mathrm{F}(G) = C_{p_1}$, we have
	\[
	C_{p_1} \cong \frac{G}{\mathrm{F}(G)} \cong \frac{G/C_m}{\mathrm{F}(G)/C_m} \cong \frac{P}{C_{p_1^{\alpha_1-1}}}.
	\]
	This shows that $P$ has a cyclic subgroup of order $p_1^{\alpha_1-1}$. Now, \cite[Chapter~$5$, Theorem~$4.4$]{Go} gives a classification of such groups. In particular, using the fact that $P$ is itself a $\mathrm{CCS}$ group or a characteristically simple group, we obtain that $P \in \{C_{p_{1}} \times C_{p_{1}}, C_{p_{1}^{\alpha_{1}}},D_{2^{\alpha_{1}}},Q_{2^{\alpha_{1}}} \}$. Let $p = p_1$ and $\alpha=\alpha_1$. In the following, for each of the possibility of $P$, we give necessary and sufficient conditions to the group $G = C_m \rtimes P$ in order to be a $\mathrm{CCS}$ group.
	\subsection{The case $P = C_p \times C_p$}
	\mbox{}\\
	Let $G = C_m \rtimes (C_p \times C_p)$. From Theorem~\ref{sec2:thm:1}, $\mathrm{F}(G)$ is a maximal subgroup of $G$ of order $mp$. Take $H \leq \mathrm{F}(G)$ of order $p$. Then, $H \leq Q$, for some Sylow $p$-subgroup $Q$ of $G$, with $Q = \langle a,b \, | \, a^p=b^p = 1, ab=ba \rangle$ and $H = \langle a \rangle$. Thus, we may write $G = C_m \rtimes_{\psi} Q$, where $C_m = \langle c \rangle$, for some homomorphism $\psi : Q \to \mathrm{Aut}(C_m)$, where $\psi(a)(c) = c^t$ and $\psi(b)(c) = c^r$, with $(t,m)=(r,m)=1$, $t^p \equiv 1 \mod m$ and $r^p \equiv 1 \mod m$. Observe now that $\mathrm{F}(G) = C_{mp} = \langle c,a \rangle$, and therefore we may suppose $t=1$. In conclusion, we have that 
	\[
	G = \langle c,a,b \ | \ c^{m}=a^{p}=b^{p}=1,ab=ba,ac=ca,bcb^{-1}=c^{r} \rangle.
	\]
	Observe finally that putting $x=ac=ca$ and $y=b$, we have 
	\[
	G = \langle x,y \, | \, x^{mp} = y^p = 1, yxy^{-1} = x^k \rangle,
	\]
	where $k\equiv r \mod m$, $k^p \equiv 1 \mod mp$, $k \not \equiv 1 \mod mp$.
	Before characterizing $\mathrm{CCS}$ group of the form above, we need a number theoretic lemma.
	\begin{lemma}\label{lemma:numeric}
		Let $m$ be an integer number, $p$ a prime number with $p<q$ for every prime $q$ dividing $m$. Let $k$ be an integer with $(m,k)=1$, $k \not \equiv 1 \mod m$, $k^p \equiv 1 \mod m$. Then, 
		\[
		(k-1,m)=1
		\]
		if, and only if,
		\[
		(k^u-1,m)=1
		\]
		for every $u \in \{1,\dots,p-1\}$.
	\end{lemma}
	\begin{proof}
		If $(k^u-1,m)=1$ for every $u \in \{1,\dots,p-1\}$ then $(k-1,m)=1$ trivially holds.\\
		Suppose that $(k-1,m)=1$. Aiming for a contradiction, suppose that there exists $u \in \{2,\dots,p-1\}$ for which $(k^u-1,m)>1$, and let $u$ be minimal with this property. Take $q$ to be a prime diving $(k^u-1,m)$. We have that
		\[
		q | k^u-1 = (k-1)(1+k+\cdots+k^{u-1}).
		\]
		Since $q$ divides $m$, and $(k-1,m)=1$, we have that $q$ divides $1+k+\cdots+k^{u-1}$. But $m$ divides $k^p-1$, and therefore $q$ also divides $1+k+\cdots+k^{p-1}$. Thus,
		\[
		q|(1+k+\cdots+k^{p-1}) - (1+k+\cdots+k^{u-1}) = k^u(1+k+\cdots+k^{p-1-u}).
		\]
		Since $(m,k)=1$, we have that $q|(1+k+\cdots+k^{p-1-u})$. Thus, $q| k^{p-u}-1$. By minimality of $u$, we obtain that $u \leq p-u$, and so $u \leq p/2$.\\
		Now $q$ divides $1+k+\cdots+k^{u-1}$ and also $1+k+\cdots+k^{p-u-1}$. Thus,
		\[
		q|(1+k+\cdots+k^{u-1})-(1+k+\cdots+k^{p-u-1}) = k^u(1+k+\cdots+k^{p-2u-1}).
		\]
		So that $q$ divides $k^{p-2u}-1$, and by minimality of $u$ we obtain $u \leq p/3$. Continuing in this way, we obtain $u \leq p/t$ for every $t$, a contradiction.
	\end{proof}
	\begin{proposition}\label{prop:CpxCp}
		Let 
		\[
		G = \langle x,y \, | \, x^{mp} = y^p = 1, yxy^{-1} = x^k \rangle,
		\]
		where $m$ is an integer, $p$ the smallest prime dividing the order of $G$, $(m,k)=1$, $k^p \equiv 1 \mod mp$, $k \not \equiv 1 \mod mp$. Then, $G$ is a $\mathrm{CCS}$ group if, and only if,
		\[
		(k-1,m)=1.
		\]
	\end{proposition}
	\begin{proof}
		Observe that since $k^p-1 \equiv 0 \mod mp$, we have that $k^p-1\equiv0 \mod m$ and $k^p-1\equiv0 \mod p$. By Fermat's Little Theorem, we have that $k^p \equiv k \mod p$, so that $k\equiv1 \mod p$. Thus, since $k \not \equiv 1 \mod mp$ and $(m,p)=1$, we have that $k\not \equiv 1 \mod m$ and $k^p \equiv 1 \mod m$.\\
		
		Suppose now that $G$ is a $\mathrm{CCS}$ group. Aiming for a contradiction, suppose that there exists a prime $q$ dividing $(k-1,m)$. Take the subgroup
		\[
		H = \langle x^q,y \rangle.
		\]
		Let $\varphi \in \mathrm{Aut}(G)$. Since $\langle x \rangle\, \mathrm{char} \,G$, we have that $\varphi(x) = x^r$, with $(r,mp)=1$. Let now $\varphi(y) = x^a y^b$, for some integers $a,b$. Observe that, since $\varphi(yxy^{-1}) = \varphi(x^k)$, we have
		\[
		x^a y^b x^r y^{-b} x^{-a} = x^{kr},
		\]
		that is
		\[
		x^{rk^b} = x^{rk}.
		\]
		This means that $rk^b \equiv rk \mod mp$. Since $(r,mp)=1$, we obtain $k^b \equiv k \mod mp$. Multiplying both sides by $k^{p-1}$, we obtain that $k^{b-1}\equiv1 \mod mp$, so that $b\equiv1 \mod p$. Thus, we may suppose that $\varphi(y) = x^a y$.\\
		Next, observe that $\varphi(y)^p = 1$, that is,
		\[
		x^{a(1+k+\cdots+k^{p-1})} = 1,
		\]
		so that $q|mp|a(1+k+\cdots+k^{p-1})$. But $q$ divides $k-1$, meaning that $k \equiv 1 \mod q$, and so
		\[
		1+k+\cdots+k^{p-1} \equiv p \mod q.
		\]
		Since $p<q$, we obtain that $q|a$. In particular, $\varphi(y) = x^{qt}y \in H$. This shows that $H$ is a characteristic subgroup of $G$, and then cyclic.  This implies that $\langle y \rangle$ is a normal subgroup of $G$, and therefore $G$ would be cyclic, a contradiction. Therefore, $(k-1,m)=1$.\\
		Conversely, suppose that $(k-1,m)=1$. By Lemma~\ref{lemma:numeric}, we have that $(k^u-1,m)=1$ for every $u \in \{1,\dots,p-1\}$. Now, since $k^p \equiv 1\mod m$, we have that $k^{up}-1 \equiv 0 \mod m$, and so
		\[
		k^{up}-1 = (k^u-1)(1+k^u+\cdots+k^{u(p-1)}) \equiv 0 \mod m.
		\]
		Since $(k^u-1,m)=1$, we obtain
		\[
		1+k^u+\cdots+k^{u(p-1)} = 0 \mod m.
		\]
		for every $u \in \{1,\dots,p-1\}$. Observe now that since $k=1 \mod p$, we have 
		\[
		1+k^u+\cdots+k^{u(p-1)} = p \equiv 0 \mod p,
		\]
		and thus
		\[
		1+k^u+\cdots+k^{u(p-1)} \equiv 0 \mod mp.
		\]
		Let now $g \in G \setminus \langle x \rangle$. Then, $g = x^ay^b$, with $b \not \equiv 0 \mod p$. Then,
		\[
		g^p = x^{a(1+k^b+\cdots+k^{b(p-1)})} = 1.
		\]
		Therefore, every element in $G$ which is not contained in $\langle x \rangle$ has order $p$. Now let $H$ be a subgroup of $G$. If $H \leq \langle x \rangle$, of course $H$ is cyclic. Suppose that $H \not \leq \langle x \rangle$, so that $ H = \langle x^d, x^ay^b\rangle$. Let $\varphi: G \to G$ sending $x$ to $x$ and $x^ay^b$ to $y$. This defines an automorphism of $G$, since $x^ay^b$ has order $p$, and it is easy to see that it is bijective. \\
		Now if $a \neq 0 \mod mp$ we have $\varphi(H) = \langle x^d,y \rangle \neq \langle x^d,x^ay^b \rangle$, and therefore $H$ is not characteristic. If $a=0$, then $H = \langle x^d,y^b \rangle = \langle x^d, y \rangle$. Consider now the automorphism $\psi$ of $G$ sending $x$ to $x$ and $y$ to $xy$. Again this is well defined. Suppose that $H$ is characteristic. Then, $\langle x^d,y \rangle = \langle x^d,xy \rangle$, and this implies that $y \in \langle x^d,xy \rangle$, a contradiction. Thus, every characteristic subgroup of $G$ is contained in $\langle x \rangle$, and therefore it is cyclic.
	\end{proof}
	\subsection{The case $P = C_{p^{\alpha}}$}
	\mbox{}\\
	Suppose that $G \cong C_{m} \rtimes C_{p^{\alpha}}$. From Theorem~\ref*{sec2:thm:1}, we know that $\mathrm{F}(G)$ is a maximal subgroup of $G$ of order $mp^{\alpha-1}$. Take $H \le \mathrm{F}(G)$ of order $p^{\alpha-1}$. Then $H \le Q$ for some Sylow $p$-subgroup $Q$ of $G$, with $Q = \langle y \rangle$ and $H =  \langle y^p \rangle$. Thus, $G = C_m \rtimes_{\psi} Q$, for some homomorphism $\psi : Q \to \mathrm{Aut}(C_m)$, with $\psi(y)(x) = x^k$, where $x$ is a generator of $C_m$. In particular, $k^{p^\alpha} = 1 \mod m$. In other words, we have that
	\[
	G = \langle x,y \, | \, x^m = y^{p^\alpha} = 1, \, yxy^{-1} = x^k\rangle.
	\]
	\begin{proposition}
		Let $m,k,p,\alpha$ integers with $p$ prime, $(m,k)=(m,p)=1$ and $k^{p^\alpha} \equiv 1 \mod m$. Consider the group
		\[
		G =\langle x,y \ | \ x^{m}=y^{p^{\alpha}}=1, yxy^{-1}=x^{k} \rangle.
		\]
		Then $G$ is a $\mathrm{CCS}$ group if, and only if, $(m,k-1)=1$ and $k^p \equiv 1 \mod m$.
	\end{proposition}
	\begin{proof}
		If $(m,k-1)=1$ and $k^p \equiv 1 \mod m$, then $G$ is an $\mathrm{NCS}$ group (see \cite[Corollary~$1.8$]{Bia} or~\cite[Theorem~$1.1$]{DM}). Therefore, $G$ is also a $\mathrm{CCS}$ group.\\
		Suppose now that $G$ is a $\mathrm{CCS}$ group. Observe that $\mathrm{F}(G) = \langle x,y^p \rangle$. In particular, $x$ and $y^p$ commute, that is $x^{k^p} = x$. This implies that $k^{p} \equiv 1 \mod m$.\\ Arguing as in the proof of Proposition~\ref{prop:CpxCp}, we see that if there exists a prime $q$ dividing $(m,k-1)$, then the subgroup $H = \langle x^q,y \rangle$ is a characteristic subgroup of $G$. Thus, $H$ must be cyclic, and so $\langle y \rangle$ is a normal subgroup of $H$. Therefore, $\langle y \rangle$ is normal in $G$, implying that $G$ is cyclic, a contradiction.
	\end{proof}
	
	\subsection{The case $P = D_{2^\alpha}$}
	\mbox{}\\
	Suppose that $G \cong C_m \rtimes D_{2^\alpha}$. By Theorem~\ref{sec2:thm:1}, we know that $\mathrm{F}(G)$ is a maximal subgroup of $G$ of order $m2^{\alpha-1}$. Take $H \le \mathrm{F}(G)$ of order $2^{\alpha-1}$. Then $H \leq Q$ for some Sylow $2$-subgroup $Q$ of $G$, with $Q = \langle r,s \, | \, r^{2^{\alpha-1}} = s^2 = 1, srs=r^{-1}\rangle$, and $H = \langle r \rangle$ (since $H$ is a cyclic maximal subgroup of the dihedral group $Q$, we can not have $s \in H$). Thus, $G = C_m \rtimes_{\psi} Q$, for some homomorphism $\psi : Q \to \mathrm{Aut}(C_m)$, with $\psi(r)(c) = c^t$ and $\psi(s)(c) = c^\ell$, where $c$ is a generator of $C_m$. Observe now that $\mathrm{F}(G) = C_{2^{\alpha-1}m} = \langle c, r \rangle$, therefore $c$ and $r$ commute, implying that $t=1$. Moreover, $\psi(s)$ has to have order $2$, and thus $\ell = -1$. In conclusion, we have
	\[
	G = \langle c,r,s \ | \ c^{m}=r^{2^{\alpha-1}}=s^{2}=1,srs=r^{-1},cr=rc,scs=c^{-1} \rangle.
	\]
	Let $x=cr^{2}$ and $y=s$. An easy computation shows that $x^{m2^{\alpha-1}}= y^2 = 1$ and that  $xy=yx^{-1}$. Therefore,
	\[
	G = \langle x,y \, | \, x^{m2^{\alpha-1}}= y^2 = 1, yxy = x^{-1}\rangle.
	\] 
	This shows that $G$ is the dihedral group of order $2^\alpha m$.\\
	Conversely, it is a well known fact that all the characteristic subgroup of any dihedral group are contained in the cyclic group generated by the rotation, and therefore all of the dihedral groups are $\mathrm{CCS}$ groups. 
	\subsection{The case $P = Q_{2^\alpha}$}
	\mbox{}\\
	Suppose that $G \cong C_{m} \rtimes_{\varphi} Q_{2^{\alpha}}$. Since $m$ is odd, $G$ is a split extension of a cyclic group and a quaternion group. In other words, $G$ is a dicyclic group. \\
	Conversely, it is a well known fact that each characteristic subgroup of a dicyclic group generated by $x$ and $y$, with $y$ being the generator of order $2$, is contained in $\langle x \rangle$, and therefore it is a $\mathrm{CCS}$ group.
	\subsection{Conclusion} 
	The discussion above allows us to obtain the following characterization of $\mathrm{CCS}$ non-nilpotent group $G$ with $G' < G$.
	
	\begin{proposition}
		Let $G$ be a non-nilpotent $\mathrm{CCS}$ group with $G^{'} \ne G$ and with $|G|= p_1^{\alpha_1}p_2^{\alpha_2}\cdots p_t^{\alpha_t}$, where $p_1 < p_2 < \dots < p_t$ are distinct primes, $\alpha_i$ are non-negative integers and $t>1$. Set $p=p_1$, $\alpha= \alpha_1$ and $m=p_2^{\alpha_2}\cdots p_t^{\alpha_t}$. Then  one of the following occurs. 
		\begin{enumerate}[label=\roman*.]
			\item $G$ has presentation
			\[
			\langle x,y \, | \, x^{mp} = y^p = 1, yxy^{-1} = x^k \rangle,
			\]
			where $m,k$ are positive integers, $p$ the smallest prime dividing the order of $G$, $(m,p)=1$, $k^p \equiv 1 \mod mp$, $k \not \equiv 1 \mod mp$, and $(k-1,m)=1$.
			\item $G$ has presentation 
			\[  \langle x,y \ | \ x^m=y^{p^{\alpha}}=1,yxy^{-1}=x^{k} \rangle,
			\]
			where $m,k$ are positive integers, $p$ the smallest prime dividing the order of $G$, $(k-1,m)=1$ and $k^{p} \equiv 1 \mod m$. 
			\item $p=2$, and $G$ is either the dicyclic group of order $2^\alpha m$, or $G$ is the dihedral group of order $2^\alpha m$.
		\end{enumerate}
		Moreover, each of these groups is a $\mathrm{CCS}$ group.
	\end{proposition}
	\section{Perfect $\mathrm{CCS}$ groups}\label{sec:3}
	
	Suppose that $G$ is a perfect $\mathrm{CCS}$ group. We require the following known lemma, of which we report the very short proof.
	\begin{lemma}\label{LemmaPerf1}
		Let $G$ be a perfect group and let $N \trianglelefteq G$ cyclic. Then, $N \leq Z(G)$.
	\end{lemma}
	\begin{proof}
		Since $N$ is normal in $G$, we have $\mathbf{N}_G(N)/\mathbf{C}_G(N) = G/\mathbf{C}_G(N) \leq \mathrm{Aut}(N)$. $N$ is cyclic, so its automorphism group is abelian. Therefore, $G/\mathbf{C}_G(N)$ is abelian, implying that $\mathbf{C}_G(N) \geq G' = G$, so that $\mathbf{C}_G(N) = G$, which is $N \leq Z(G)$.
	\end{proof}	
	To continue our analysis, we need some basic results about perfect central extensions and the Schur multiplier. \\
	Recall that a central extension of a group $G$ is a pair $(H,\alpha)$ such that we have the following short exact sequence:
	\[
	1 \to Z \to H \xrightarrow{\alpha} G \to 1,
	\]	
	where $Z \leq Z(H)$. If $H$ is a perfect group, then the central extension is said to be perfect.\\
	If $G$ is a perfect group, then $G$ admits a special central extension, called \emph{universal central extension} of $G$, denoted with $(\tilde{G},\pi)$. The kernel of $\pi$ is called the \emph{Schur multiplier} of $G$, and it is denoted by $H^2(G,\mathbb{Z})$. This extension has the property that for any other perfect central extension of $G$, say $(H,\varphi)$, $\ker \varphi$ is a quotient of $H^2(G,\mathbb{Z})$. We refer the reader to \cite{As} for a detailed discussion of this material.\\
	We report another known lemma, which will be used in the proof of the upcoming theorem. For a reference, see~\cite{Sch}.
	\begin{lemma}
		Let $G,H$ be two finite groups. Then,
		\[
		H^2(G \times H,\mathbb{Z}) \cong (H^2(G,\mathbb{Z})\times H^2(H,\mathbb{Z}))\times\left(\frac{G}{G'}\otimes \frac{H}{H'}\right).
		\]
		In particular, if $G$ and $H$ are perfect groups, then
		\[
		H^2(G \times H,\mathbb{Z}) \cong H^2(G,\mathbb{Z})\times H^2(H,\mathbb{Z}).
		\]
	\end{lemma}
	
	\begin{theorem}
		Let $G$ be a finite perfect group. Then $G$ is a $\mathrm{CCS}$ group if, and only if, $Z(G)$ is cyclic and it is the unique maximal characteristic subgroup of $G$. \\
		Moreover, if this happens, $G$ fits into a short exact sequence of the form $$1 \rightarrow C \rightarrow G \rightarrow S \times \cdots \times S$$ where $S$ is a non-abelian simple group and $C$ is a cyclic group isomorphic to a quotient of $H^{2}(S,\mathbb{Z}) \times \cdots \times H^{2}(S,\mathbb{Z})$.
	\end{theorem}
	\begin{proof}
		Suppose first that $G$ is a $\mathrm{CCS}$ group. Let $N$  be a characteristic subgroup of $G$. Then, $N$ is cyclic, and by Lemma~\ref{LemmaPerf1} it is contained in the center. Thus, $Z(G)$ is the unique maximal characteristic subgroup of $G$. The viceversa is trivially true.\\
		Suppose now that $G$ is a $\mathrm{CCS}$ group. Since $Z(G)$ is the unique maximal characteristic subgroup of $G$, $G/Z(G)$ is characteristically simple, that is
		\[
		\frac{G}{Z(G)} \cong S \times \cdots \times S \rightarrow 1,
		\]
		for some finite simple group $S$. In particular, we have a short exact sequence of the form
		\[
		1 \rightarrow Z(G) \rightarrow G \rightarrow S \times \cdots \times S.
		\]
		Note that $S$ is not abelian, since otherwise $G$ is solvable. Since $S$ is non-abelian simple, it is perfect, and so is $S \times \cdots \times S$. Thus, this admits the universal central extension, say $(A, \pi)$. Then we have an exact sequence: $$ 1 \rightarrow H^{2}(S \times \cdots \times S,\mathbb{Z}) \rightarrow A \rightarrow S \times \cdots \times S \rightarrow 1.$$ Since also $G$ is a perfect central extension of $S \times \cdots \times S$ we get that $$Z(G) \cong H^{2}(S \times \cdots \times S,\mathbb{Z})/N \cong [H^{2}(S,\mathbb{Z}) \times \cdots \times H^{2}(S,\mathbb{Z})]/N.$$ 
	\end{proof}		
	For example, suppose that $G$ is a $\mathrm{CCS}$ group that arises as a perfect central extension of $ S \times S $, where $ S = A_5$. Since $H^2(S, \mathbb{Z}) \cong C_2$, the group $G$ fits into a short exact sequence of the form
	\[
	1 \rightarrow \frac{C_2 \times C_2}{N} \rightarrow G \rightarrow A_5 \times A_5 \rightarrow 1,
	\]
	for some normal subgroup $ N \trianglelefteq C_2 \times C_2$, where $ (C_2 \times C_2)/N = Z(G) $. \\
	Because $ G $ is a  $\mathrm{CCS}$ group, its center $ Z(G) $ must be cyclic. This implies that $ N \cong C_2 $ and therefore $ Z(G) \cong C_2 $.
	 A group of this type appears in the \texttt{PerfectGroup} library of \textsc{GAP}~\cite{GAP}, with identifier $(7200, 2)$.

\section{Proofs of Theorems \ref{Th2} and \ref{thm:skbr}}\label{sec:skew}
Let $(B,+,\cdot)$ be a skew left braces  as in the Introduction. The additive and multiplicative groups of $B$ are related by the so-called \emph{lambda map}:
\[
\lambda \colon (B, \cdot) \longrightarrow \mathrm{Aut}(B,+), \quad a \mapsto \lambda_a,
\]
where, for each $a \in B$,
\[
\lambda_a \colon B \longrightarrow B, \quad b \mapsto -a + a \cdot b.
\]
It is shown in \cite{GV17} that $\lambda$ is a well-defined group homomorphism. The importance of this map lies in the fact that it allows one to express the additive operation in terms of the multiplicative one and vice versa:
\[
a + b = a \cdot \lambda_a(b), \qquad a \cdot b = a + \lambda_a(b).
\]
A subset $S \subseteq B$ is called a \emph{subbrace} of $B$ if $S$ is a subgroup of both $(B,+)$ and $(B,\cdot)$. In that case we write $S \le B$. Note that in this way the triple $(S,+,\cdot)$ is a skew brace. Finally note that in a skew left brace $(B,+,\cdot)$, the identity element of the additive group $(B,+)$ coincides with the identity element of the multiplicative group $(B,\cdot)$. Indeed let $0$ denote the identity element in the additive group and $1$ the identity element in the  
multiplicative group. From the brace identity $a\cdot(b+c)=a\cdot b-a+a\cdot c$, setting $b=c=0$ we obtain
\[
a\cdot 0 = a\cdot 0 - a + a\cdot 0.
\]
Simplifying, it follows that $a=a\cdot 0$ for all $a\in B$. In particular, taking $a=1$ yields $1=1\cdot 0=0$, and hence the identity elements of $(B,+)$ and $(B,\cdot)$ coincide.

\begin{lemma} \label{lemmachar}
Let $(B,+,\cdot)$ be a skew brace and $H$ a characteristic subgroup of $(B,+)$. Then $(H,\cdot) \le (B,+)$. In particular $H$ is a subbrace of $B$.
\end{lemma}

\begin{proof}
	Recall that for every $b \in B$, the map $
	\lambda_b \in \operatorname{Aut}(B,+)$.
	Since $H$ is characteristic in $(B,+)$, it follows that for all $b \in B$ and $h \in H$, $\lambda_b(h) \in H$. Take $h,h' \in H$. By the definition of the $\lambda$-map we have
	$h \cdot h' = h + \lambda_h(h')$, which is therefore an element of $H$. Moreover:
	\[
	\lambda_h(h^{-1})=-h+h\cdot h^{-1}=-h+1.
	\]
	Since the additive and multiplicative identities of $B$ coincide, $1=0$, it follows that
	\[
	\lambda_h(h^{-1})=-h.
	\]
	Applying $\lambda_h^{-1}$ to both sides yields
	\[
	h^{-1}=\lambda_h^{-1}(-h),
	\]
Since $-h\in H$ and $\lambda_h^{-1}$ is again an automorphism of $(B,+)$, it follows that $h^{-1}\in H$.

\end{proof}

A subset $I \subseteq B$ is called an \emph{ideal} if it is a normal subgroup of both $(B,+)$ and $(B,\cdot)$, and 
\[
\lambda_a(I) \subseteq I \quad \text{for all } a \in B.
\]
In that case we write $I \trianglelefteq B$. The importance of ideals lies in the fact that they allow one to consider quotients of skew braces: The \emph{quotient skew brace}  $B/I$ is defined as the set 
\[
B/I = \{ b + I = b \cdot I \mid b \in B \}.
\]

The operations on $B/I$ are defined by
\[
(b + I) + (b' + I) := (b + b') + I,
\]
\[
(b + I) \cdot  (b' + I) := (b \cdot b') + I,
\]
for all $b, b' \in B$. With these operations, $B/I$ becomes a skew brace, called the \emph{quotient skew brace of $B$ modulo $I$}. Note that $(B/I,+) \simeq (B,+)/(I,+)$ and $(B/I,\cdot) \simeq (B,\cdot)/(I,\cdot)$.

The notions of homomorphism and isomorphism between skew braces are the natural extensions of the corresponding notions for groups. In particular, a homomorphism of skew braces is a map preserving both the additive and the multiplicative structures. As in group theory, isomorphisms are bijective homomorphisms.
Moreover , the classical isomorphism theorems extend to the setting of skew braces: kernels, images, and quotients by ideals behave analogously to the group-theoretic case, and the First, Second, and Third Isomorphism Theorems hold in this context.

 We say that $B$ is a \emph{two-sided skew brace} if, for all $a,b,c \in B$, 
\[
(a+b)\cdot c = a \cdot c - c + b \cdot c.
\]

For any $a,b \in B$, we define the \emph{star product}
\[
a * b := \lambda_a(b) - b = -a + a\cdot b - b.
\]
Intuitively, $*$ measures the difference between the additive and multiplicative structures of $B$. The skew brace $B$ is called \emph{trivial} if $a*b=0$ for all $a,b \in B$, which is equivalent to $a+b = a\cdot b$ for all $a,b \in B$. In particular, if $B$ has prime order $p$, then $B$ is trivial.
For subsets $X,Y \subseteq B$, we define
\[
X * Y := \langle x * y \mid x \in X,\, y \in Y \rangle_+,
\]
i.e., the additive subgroup generated by all elements $x*y$.  
For example, $B^2 := B*B$ is an ideal of $B$ and is the minimal ideal such that the quotient $B/(B*B)$ is a trivial skew brace (see \cite[Proposition 2.3]{Cedo2018}).

\begin{proof}[Proof of Theorem \ref{Th2}]
	Suppose the theorem is not true and let $(B,+,\cdot)$ be a minimal counterexample with respect to the order of $B$.
	First, suppose that $(B,+)$ is not characteristically simple. Let $K$ be a proper non trivial characteristic subgroup of $(B,+)$. By Lemma \ref{lemmachar}  $K$ is a subbrace of $B$, and so by induction we have that $(K,+)$ is cyclic. It follows that $(B,+)$ is a $\mathrm{CCS}$ group. Since $|B|$ is an odd prime power, by Theorem \ref{thm:main} we deduce that
	\[
	(B,+) = p_{+}^{1+2n}.
	\] 
	Therefore, $(B,+)$ admits a unique characteristic subgroup, namely its center
	\[
	Z(B,+) = C_p.
	\] 
	Since $(B,\cdot)$ is abelian, $Z(B,+)$ is an ideal of $B$. Hence, by induction,
	\[
	(B,+)/Z(B,+)
	\] 
	is cyclic, and it follows that $(B,+)$ is abelian, a contradiction.
	Therefore, $(B,+)$ is characteristically simple, meaning that 
	\[
	(B,+) = (C_p)^n.
	\] 
	Since $(B,\cdot)$ is abelian, $B$ is a two-sided left brace, and by \cite[Theorem 2.1]{Trappeniers2023}, $B*B$ is properly contained in $B$. 
	 If $B*B = 1$, then $B$ is a trivial brace, and we would have that $(B,+)$ is cyclic. Hence, the order of $B/(B*B)$ is strictly smaller than $|B|$, and by induction we get that
	\[
	(B,+)/(B*B)
	\] 
	is cyclic. Since $(B,+)$ is elementary abelian, we obtain
	\[
	(B,+)/(B*B) = C_p \quad \text{and} \quad B*B = C_p.
	\] 
	Therefore, $|B| = p^2$, and by \cite[Proposition 2.4]{Bachiller2015} we get that $(B,+)$ is cyclic, a contradiction.
	
\end{proof}

Finally, we conclude by proving Theorem \ref{thm:skbr}.

\begin{proof}[Proof of Theorem \ref{thm:skbr}]
	Since $\mathrm{F}(B,+)$, the fitting subgroup of $(B, +)$, is a proper characteristic subgroup of $(B,+)$, it follows from Lemma \ref{lemmachar} that $(\mathrm{F}(B,+),+,\cdot)$ is a sub-skew brace of $(B,+,\cdot)$. By assumption, $(B,+)$ is a non-nilpotent and non-perfect $\mathrm{CCS}$ group. Therefore, by Theorem~\ref{sec2:thm:1}, $(\mathrm{F}(B,+),+)$ is a cyclic group of order $mp^{\alpha-1}$. Let $P_i$ be a Sylow $p_i$-subgroup of $(B,+)$, with $i \in \{2,\dots,t\}$. Since $(B,+)$ is supersolvable, by Lemma~\ref{lemmasupersolv2} the subset of elements of order coprime to $p$ is a characteristic subgroup of $(B,+)$, and so it is contained in $\mathrm{F}(B,+)$. In particular, $P_i \leq \mathrm{F}(B,+)$. Hence, $(P_i,+)$ is a cyclic characteristic subgroup of $(B,+)$, and consequently, by Lemma \ref{lemmachar}, $(P_i,+,\cdot)$ is a sub-skew brace of $(B,+,\cdot)$. Since $(P_i,+)$ is cyclic, from \cite[Corollary of Theorem 2]{Rump2007B} implies that $(P_i,\cdot)$ is cyclic as well. Therefore, the Sylow $p_i$-subgroups of $(B,\cdot)$ are cyclic. 
	Suppose now that $(B,\cdot)$ is nilpotent. Then
	\[
	(B,\cdot) \cong C_{p_2^{\alpha_2}} \times \cdots \times C_{p_t^{\alpha_t}} \times Q 
	\;\cong\; C_m \times Q,
	\]
	where $(Q,\cdot)$ is a Sylow $p$-subgroup of $(B,\cdot)$. Let $H$ be a subgroup of $F(B,+)$ of order $p^{\alpha-1}$. Since $H$ is characteristic in $F(B,+)$, it is also characteristic in $(B,+)$. In particular, by Lemma \ref{lemmachar}, $(H,+,\cdot)$ is a sub-skew brace of $(B,+,\cdot)$. Since $(H,+)$ is cyclic of odd prime power order, by \cite[Corollary of Theorem 2]{Rump2007B}  $(H,\cdot)$ is cyclic. But $H$ is contained in a Sylow $p$-subgroup of $(B,\cdot)$, which has order $p^\alpha$. This completes the proof in the nilpotent case. 
Assume now that $(B,\cdot)$ is non-nilpotent. Since $(\mathrm{F}(B,+),+,\cdot)$ is a sub-skew brace, it follows that $(\mathrm{F}(B,+),\cdot)$ is a subgroup of $(B,\cdot)$ of index $p$, and hence it is normal. Moreover, a similar argument as in the nilpotent case shows that $(\mathrm{F}(B,+),\cdot)$ is a $\mathrm{Z}$-group. In particular, $(\mathrm{F}(B,+),\cdot)$ is supersolvable. 
Therefore, the subset $K$ of elements of $(\mathrm{F}(G),\cdot)$ whose order is coprime to $p$ forms a characteristic subgroup of $(\mathrm{F}(B,+),\cdot)$ of order $m$. Since $K$ is characteristic in a normal subgroup, we deduce that $K \trianglelefteq (B,\cdot)$. Thus,
	\[
	(B,\cdot) \cong K \rtimes Q,
	\]
	where $Q$ is a Sylow $p$-subgroup of $(B,\cdot)$. Finally, note that
	\[
	C_p \;\cong\; \frac{(B,\cdot)}{(\mathrm{F}(B,+),\cdot)} 
	\;\cong\; \frac{(B,\cdot)/K}{(F(B,+),\cdot)/K} 
	\;\cong\; \frac{Q}{C_{p^{\alpha-1}}},
	\]
	which shows that $Q$ contains a maximal cyclic subgroup.
\end{proof}

\end{document}